\documentclass[preprint,12pt]{elsarticle}

\usepackage{tikz}
\newcommand{\tikznode}[2]{%
\ifmmode%
\tikz[remember picture,baseline=(#1.base),inner sep=0pt] \node (#1) {$#2$};%
\else
\tikz[remember picture,baseline=(#1.base),inner sep=0pt] \node (#1) {#2};%
\fi}
\usepackage{amssymb}
\usepackage{amsmath}
\usepackage{amsthm}
\numberwithin{equation}{section}

\journal{Linear Algebra Appl.}

\newtheorem{theorem}{Theorem}[section]
\newtheorem{lemma}[theorem]{Lemma}

\newtheorem{corollary}[theorem]{Corollary}

 \theoremstyle{definition}
 \newtheorem{definition}[theorem]{Definition}
 \newtheorem{example}[theorem]{Example}
 \newtheorem{remark}[theorem]{Remark}
\begin{document}

\begin{frontmatter}



\title{Boundary value problems for linear differential-algebraic equations: solvability via the Kronecker canonical form}
\nonumnote{%
\includegraphics[height=7.0mm]{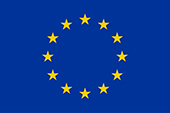}\,
This paper is supported by European Union’s Horizon 2020 research and
innovation programme under the Marie Sklodowska-Curie grant agreement
ID: 873071, project SOMPATY (Spectral Optimization: From Mathematics
to Physics and Advanced Technology).}

\author{Anar Assanova\fnref{label1}}
\author{Carsten Trunk\fnref{label2}}
\author{Roza Uteshova\fnref{label1}}
\author{Henrik Winkler\fnref{label2}}
\affiliation[label1]{organization={Institute of Mathematics and Mathematical Modeling},
            addressline={Shevchenko Str. 28}, 
            city={Almaty},
            postcode={050010}, 
            country={Kazakhstan}}
\affiliation[label2]{organization={Institut für Mathematik, TU Ilmenau},
            addressline={Weimarer Straße 25}, 
            city={Ilmenau},
            postcode={98693}, 
            country={Germany}}

\begin{abstract}
The solvability of two-point boundary value problems for constant-coefficient differential-algebraic equations is investigated. Unlike previous studies, which often assume that the matrix pair associated with the equation is regular, we consider the general singular case. Using the Kronecker canonical form, we decompose the problem into simpler subsystems, enabling a systematic analysis of solvability. The method of parameterization reduces the boundary value problem to a system of algebraic equations. We derive criteria for the existence and uniqueness of solutions and provide a comprehensive framework for solving such boundary value problems. Several illustrative examples are presented and show the applicability of the results.
\end{abstract}


\end{frontmatter}

\section{Introduction}

We consider the linear differential-algebraic equation with constant coefficients of the form
\begin{equation}\label{dae}
E\dot{x}(t)=Ax(t)+f(t), \quad t\in (0,T),
\end{equation}
subject to the boundary condition
\begin{equation}\label{BC}
Bx(0)+Cx(T)=d.
\end{equation}
Here $E, A\in \mathbb{R}^{m\times n}$, $B,C\in \mathbb{R}^{l\times n}$, $d\in\mathbb{R}^l$, $T>0$, and $f\in C^1((0,T),\mathbb{R}^m)$.
By a solution of the boundary value problem \eqref{dae}, \eqref{BC} we mean a function $x\in C^1([0,T],\mathbb{R}^n)$ that satisfies equation \eqref{dae} and the boundary condition \eqref{BC}.

Differential-algebraic equations have become widely used as a tool to model and simulate dynamical systems with constraints in various applications \cite{Kunkel:2006,brenan1995numerical,ascher1998computer,lamour2013differential, riaza2008differential}. 

Boundary value problems for differential-algebraic equations appear in many problems from engineering such as electrical networks or multi-body problems.
This theory began to develop through the adaptation of modified shooting and collocation methods, originally designed for boundary value problems in ordinary differential equations \cite{ascher1992projected, bai1992modified, clark1989numerical, stover2001collocation},
for an overview of this development till 2001 we refer to \cite{RabierRheinboldt02}, for a more recent overview see \cite{Lamour_Maerz2015}.

The approach in \cite{Lamour_Maerz2015}, see also 
\cite{lamour2013differential,marz1996canonical, marz2004solvability},
is rather general allowing non-linear
problems or time varying coefficients which can be also
operator valued (this is, $E$ and $A$ are linear operators in some Banach space). That approach is based on the projector-based analysis, which is presented in, e.g., \cite{lamour2013differential}.

Here we restrict to the more special
case of a linear, finite-dimensional (i.e.\ $E$ and $A$ are matrices) and (time) constant case. This setting allows to base our approach on the Kronecker canonical form (for details we refer to the section below). The Kronecker canonical form was utilized to establish a solution theory for linear, time-constant 
matrix valued differential-algebraic equations, see, e.g., \cite{berger_diss,Trenn13}. However, here it is the first  time used to obtain a detailed decription for the solvability and unique solvability of the  boundary value problem \eqref{dae}, \eqref{BC}. It is the main contribution of this paper.  

In \cite{Assanova_Trunk_Uteshova_2024}, see also \cite{berger2016linear}, we studied the boundary value problem \eqref{dae}, \eqref{BC} under the assumption that the matrix pair $(E,A)$ is regular, i.e. $m=n$ and $\text{det}(\lambda E-A)$ is not the zero polynomial. Using the idea of the method of parameterization \cite{Dzhumabaev:1989}, we set a parameter as the value of the solution at the left endpoint of the interval. Applying the Weierstrass canonical form to the matrix pair associated with the differential-algebraic equation, we derived a criterion for the unique solvability of the problem.

In this paper, we address problem \eqref{dae}, \eqref{BC} in the general case, where the matrix pair $(E,A)$ in \eqref{dae} is not necessarily regular. We use the Kronecker canonical form of $(E,A)$ and the method of parameterization to obtain a criterion for the unique solvability of problem \eqref{dae}, \eqref{BC}. When applying the parameterization method to the boundary value problem, the question of solvability reduces to determining the solvability of a system of algebraic equations for the introduced parameters. 

The paper is organized as follows. In Section 2, we provide the necessary preliminaries, including an overview of the Kronecker canonical form for matrix pairs and the concept of generalized inverses, which play a central role in our analysis. Section 3 focuses on the solvability of boundary value problems for the decomposed blocks arising from the Kronecker form. In Section 4, we establish existence and uniqueness theorems for the general boundary value problem \eqref{dae}, \eqref{BC}, utilizing the method of parameterization. Finally, in Section 5, we illustrate the results with a selection of examples that demonstrate the applicability and effectiveness of our approach.

\section{Preliminaries}
    \subsection{The Kronecker canonical form for matrix pairs}
    We will employ a well-known tool for constant-coefficient differential-algebraic equations, namely, the Kronecker canonical form of the matrix pair $(E,A)$ associated with equation \eqref{dae}. The Kronecker canonical form is going back to \cite{kronecker1890algebraische}, see also \cite[Ch. XII]{gantmacher1959theory}. The form in Theorem \ref{Th_KCF} below is taken from \cite{berger2016linear}.

   Let $\alpha$ be a multi-index, $\alpha=(\alpha_1, \ldots, \alpha_{n_{\alpha}})\in \mathbb{N}^{n_{\alpha}}$. The absolute value of $\alpha$ is defined by $|\alpha|=\sum \limits_{i=1}^{{n_{\alpha}}}\alpha_i$.  For $k \in \mathbb{N}$ we
introduce the matrices
\begin{samepage}
\[
N_k=\begin{bmatrix}
    \tikznode{a}{0} & & & \\
    \tikznode{b}{1} & & & \\
    & & & \\
    & & \tikznode{c}{1} & \tikznode{d}{0}
  \end{bmatrix}\in \mathbb{R}^{k\times k},
\quad
N_\alpha=\operatorname{diag}(N_{\alpha_1},\ldots,N_{\alpha_{n_{\alpha}}})
\in \mathbb{R}^{|\alpha|\times |\alpha|}.
\]
\begin{tikzpicture}[remember picture, overlay]
  \draw (a) -- (d);
  \draw (b) -- (c);
\end{tikzpicture}
\end{samepage}

For $k \in \mathbb{N}$, $k>1$, we define
\begin{equation*}
K_k=\begin{bmatrix}
    \tikznode{e}{1} & \tikznode{f}{0} & & \\
    \\
      &  & \tikznode{g}{1} & \tikznode{h}{0}\\
  \end{bmatrix}, ~
  L_k=\begin{bmatrix}
    \tikznode{i}{0} & \tikznode{j}{1} & & \\
    \\
      &  & \tikznode{k}{0} & \tikznode{l}{1}\\
  \end{bmatrix}\in \mathbb{R}^{(k-1)\times k}.
\end{equation*}

\begin{tikzpicture}[remember picture, overlay,shorten >=1pt,shorten <=1pt]
  \draw (e) -- (g);
  \draw (f) -- (h);
  \draw (i) -- (k);
  \draw (j) -- (l);
\end{tikzpicture}

Let $\alpha$ be a multi-index with $\alpha_i>1$, $i=1,\ldots,n_{\alpha}$. We set
\begin{equation*}
 K_\alpha=\text{diag}(K_{\alpha_1},\ldots,K_{\alpha_{n_{\alpha}}}), ~ L_\alpha=\text{diag}(L_{\alpha_1},\ldots,L_{\alpha_{n_{\alpha}}})\in \mathbb{R}^{(|\alpha|-{n_{\alpha}})\times |\alpha|}.
\end{equation*}

Note that the above notion is extended to multiindices $\alpha=(\alpha_1, \ldots, \alpha_{n_{\alpha}})\in \mathbb{N}^{n_{\alpha}}$ with some (or all) entries equal to 1. For a detailed explanation, we refer to \cite{berger2016linear}. 

The following theorem is due to Kronecker \cite{kronecker1890algebraische}, a modern version is in
\cite[Ch.: XII]{gantmacher1959theory}.

\begin{theorem}\label{Th_KCF} For any pair of matrices $E,A\in \mathbb{K}^{m\times n}$, where $\mathbb{K}$ is the field of real or complex numbers, there exist invertible matrices $W \in \mathbb{K}^{m\times m}$ and $V \in \mathbb{K}^{n\times n}$ such that
\begin{equation}\label{KCF_1}
WEV=
\begin{bmatrix} I_{n_0}& 0 & 0 & 0\\
0&  N_\alpha & 0 & 0\\
0& 0 & L_\beta & 0\\
0& 0 & 0 & L_\gamma^\top\\
\end{bmatrix} ~~ \text{and}~~~~ WAV=
\begin{bmatrix} A_0& 0 & 0 & 0\\
0& I_{|\alpha|} & 0 & 0\\
0& 0 & K_{\beta} & 0\\
0& 0 & 0 & K_\gamma^\top\\
\end{bmatrix}
\end{equation}
for some $A_0\in \mathbb{K}^{n_0\times n_0}$ and multi-indices $\alpha\in \mathbb{N}^{n_{\alpha}}$, $\beta\in \mathbb{N}^{n_{\beta}}$, and $\gamma\in \mathbb{N}^{n_{\gamma}}$. The multi-indices $\alpha, \beta, \gamma$ are unique up to a permutation of their respective entries, and the matrix $A_0$ is unique up to similarity.
\end{theorem}
 The matrix pair $(WEV, WAV)$ in \eqref{KCF_1} is called \emph{quasi-Kronecker form} of the matrix pair $(E,A)$. If, in addition, the matrix $A_0$ in \eqref{KCF_1} is in Jordan canonical form, then \eqref{KCF_1} is called \emph{Kronecker canonical form}. Note that if $\mathbb{K}=\mathbb{R}$, the Jordan canonical form corresponds to the so-called real Jordan canonical form (see \cite{LancasterRodman2005}, \cite{Puche_Reis_2018_ConstantcoefficientDAE}). 

Some of the parameters $n_0$, $n_{\alpha}$, $n_{\beta}$, and $n_{\gamma}$ in \eqref{KCF_1} are possible to be equal to zero. This means that the corresponding blocks are absent in \eqref{KCF_1}. 
Note that $\beta=(1,\ldots,1)$ (respectively, $\gamma=(1,\ldots,1)$) can occur only if at least one other block is present in \eqref{KCF_1}; i.e., $n_0>0$, $n_\alpha>0$, or $n_\gamma>0$ (respectively, $n_\beta>0$).
 
In particular, if $(E,A)$ is a regular matrix pair, then the blocks $(L_{\beta}, K_{\beta})$ and $(L_{\gamma}^\top, K_{\gamma}^\top)$ do not appear in the quasi-Kronecker form. In this case, the pair
\begin{equation}\label{WCF_1}
WEV=
\begin{bmatrix} I_{n_0}& 0 \\
0&  N_\alpha
\end{bmatrix} ~~ \text{and}~~~~ WAV=
\begin{bmatrix} A_0& 0 \\
0& I_{|\alpha|}
\end{bmatrix}
\end{equation}
is called \emph{quasi-Weierstrass form} of $(E,A)$
\cite{BT12,Berger2013addition}.

\subsection{Generalized inverses}

\begin{definition}
    Let $A\in\mathbb{R}^{m\times n}$ be any matrix. A matrix $A^g\in\mathbb{R}^{n\times m}$ is called a generalized inverse of $A$ if it satisfies the condition:
\begin{equation*}
    AA^gA=A.
\end{equation*}
\end{definition}
It is well known that a generalized inverse always exists for any matrix $A\in\mathbb{R}^{m\times n}$.
This generalized inverse is sometimes also called 
$\{1\}$-inverse as it satisfies only the first of the four Penrose equations which are used for the definition of the Moore--Penrose inverse, cf.\   \cite{benisrael2006GeneralizedInverses}. 
However, it is not necessarily unique.

 \begin{lemma}\label{lemma_Solution_AE_via_GenInv}
Let \(A\in \mathbb{R}^{m\times n}\), \(b\in\mathbb{R}^m\), and let \(A^g\in\mathbb{R}^{n\times m}\) be any generalized inverse of \(A\). Then the equation
\begin{equation}\label{le}
Ax=b
\end{equation}
is solvable if and only if
\begin{equation}\label{AAgb=b}
AA^g b=b.
\end{equation}
In this case, all solution of \eqref{le} are given by
 \begin{equation}\label{Solution_via_GeneralizedInverse}
    x=A^gb+(I_n-A^gA)z,
 \end{equation}
 for arbitrary $z\in\mathbb{R}^n$.
The solution set from \eqref{Solution_via_GeneralizedInverse},
\[
\mathcal S(A^g):=\bigl\{\,A^g b+(I_n-A^gA)z:\ z\in\mathbb{R}^n\,\bigr\},
\]
is independent of the choice of the generalized inverse $A^g$.
\end{lemma}

\begin{proof}
The first part of the lemma can be found, e.g., in 
Corollary~2 \cite[p.~53]{benisrael2006GeneralizedInverses}.
It remains to show that $\mathcal S(A^g)$ is independent of the choice of the generalized inverse.
Fix some generalized inverse $A^g$. By Corollary~1 \cite[p.~52]{benisrael2006GeneralizedInverses}, any other generalized inverse $\widetilde A^g$ can be written as
\[
\widetilde A^g \;=\; A^g + Y - A^g A\,Y\,A\,A^g
\qquad\text{for some } Y\in\mathbb{R}^{n\times m}.
\]
An arbitrary $x$ from $\mathcal S(\widetilde A^g)$ has the form 
\[
x \;:=\; \widetilde A^g b + (I_n-\widetilde A^g A)\,z
\]
for some $z\in\mathbb{R}^n$. Using $AA^g b=b$ and $AA^gA=A$, we compute
\begin{align*}
x
&= \bigl(A^g + Y - A^g A Y A A^g \bigr)b
   + z - \bigl(A^g + Y - A^g A Y A A^g \bigr)A z\\
&= A^g b + (I_n - A^g A)z \;+\; \bigl( Y b - A^g A Y A A^g b - Y A z + A^g A Y A A^g A z \bigr)\\
&= A^g b + (I_n - A^g A)z \;+\; \bigl( Y - A^g A Y \bigr) \bigl( b - A z \bigr)\\
&= A^g b + (I_n - A^g A)\bigl[\, z + Y\bigl(b-A z\bigr) \,\bigr].
\end{align*}
Thus $x\in\mathcal S(A^g)$
and we have shown that $\mathcal S(\widetilde A^g)\subseteq \mathcal S(A^g)$. By symmetry (interchanging the roles of $A^g$ and $\widetilde A^g$), the reverse inclusion holds. 
\end{proof}

Note that condition \eqref{AAgb=b} is equivalent to requiring that $b$ belongs to $\text{ran}~A$.
If $A$ is square and invertible, then it has one and only one generalized
inverse which coincides with the ordinary inverse $A^{-1}$.

In addition to the generalized inverse, we use left and right inverses. These are special cases of generalized inverses and they are relevant for matrices with full column rank or full row rank, respectively.

\begin{definition}
  A matrix $A_L^{-1}\in\mathbb{R}^{n\times m}$ is a left inverse of $A$ if it satisfies
  \begin{equation*}
      A_L^{-1}A=I_n,
  \end{equation*}
where $I_n$ is the $n\times n$ identity matrix.
\end{definition}
A left inverse exists if and only if $A$ has full column rank ($\text{rank}A=n$).  It is not necessarily unique. For matrices of full column rank, 
\begin{equation}\label{Alina}
A_L^{-1} :=(A^\top A)^{-1}A^\top
\end{equation}
is an example for a left inverse.

\begin{definition}
  A matrix $A_R^{-1}\in\mathbb{R}^{n\times m}$ is a right inverse of $A$ if it satisfies
  \begin{equation*}
      AA_R^{-1}=I_m,
  \end{equation*}
where $I_m$ is the $m\times m$ identity matrix.
\end{definition}
A right inverse exists if and only if $A$ has full row rank ($\text{rank}A=m$).  It is not necessarily unique. For matrices of full row rank, 
\begin{equation}\label{right_inverse}
A_R^{-1} :=A^\top(A A^\top)^{-1}
\end{equation}
is an example for a right inverse.

\section{Block wise existence and uniqueness for the boundary value problem}

Let  $W \in \mathbb{R}^{m\times m}$ and $V \in \mathbb{R}^{n\times n}$ be invertible matrices that transform the matrix pair $(E,A)$, associated with equation \eqref{dae}, into the quasi-Kronecker form \eqref{KCF_1} with $A_0\in \mathbb{R}^{n_0\times n_0}$. By applying the transformation
\begin{equation}\label{x_trans}x(t)=V\left(x_0^\top(t), x_1^\top(t), x_2^\top(t), x_3^\top(t)\right)^{\top},\end{equation}
where the splitting corresponds to the block sizes in \eqref{KCF_1}, and multiplying both sides of  equation \eqref{dae} from the left by $W$, we obtain 
\begin{subequations}\label{111}\begin{align}
   \dot{x}_0(t)&=A_0x_0(t)+f_0(t),\label{11x_difeq}\\
N_{\alpha}\dot{x}_1(t)&=x_1(t)+f_1(t),\label{11x_alg_eq}\\
L_{\beta}\dot{x}_2 (t)&=K_{\beta}x_2(t)+
f_2(t),\label{11x_under_eq}\\
L_{\gamma}^\top\dot{x}_3(t)&=K_{\gamma}^\top x_3(t)+f_3(t),\label{11x_over_eq}
\end{align}
\end{subequations}
with
\begin{equation*}
   Wf(t)=\left(f_0^\top(t), f_1^\top(t), f_2^\top(t), f_3^\top(t)\right)^{\top},
\end{equation*}
where $f_0\in C([0,T],\mathbb{R}^{n_0})$, $f_1\in C([0,T],\mathbb{R}^{|\alpha|})$, $f_2 \in C([0,T],\mathbb{R}^{|\beta|-n_\beta})$, and $f_3\in C([0,T],\mathbb{R}^{|\gamma|})$. The transformation \eqref{x_trans} also induces a corresponding block decomposition of the boundary condition matrices:
\begin{equation*}\label{BC_blocks}
   BV=[B_0, B_1, B_2, B_3], \quad  CV=[C_0, C_1, C_2, C_3], \quad d = \sum\limits_{i=0}^3 d_i.
\end{equation*}
Here,
\[
\begin{aligned}
    &B_0, C_0 \in \mathbb{R}^{l \times n_0}, \quad
    B_1, C_1 \in \mathbb{R}^{l \times |\alpha|}, \\
    &B_2, C_2 \in \mathbb{R}^{l \times |\beta|}, \quad
    B_3, C_3 \in \mathbb{R}^{l \times (|\gamma| - n_{\gamma})}, \quad
    d_i \in \mathbb{R}^l \text{ for } i = 0,1,2,3.
\end{aligned}
\]
The total dimension is preserved:
\[
n_0 + |\alpha| + |\beta| + |\gamma| - n_{\gamma} = n.
\]
Thus, the boundary condition \eqref{BC} decomposes into the component-wise form:
\begin{equation}\label{BC_decomposed}
    B_i x_i(0) + C_i x_i(T) = d_i, \quad i = 0,1,2,3.
\end{equation}

The four equations in \eqref{111} correspond to different structural components of the system, which are commonly referred to as follows: the ODE part \eqref{11x_difeq},  the nilpotent part \eqref{11x_alg_eq}, the underdetermined part \eqref{11x_under_eq}, and
the overdetermined part \eqref{11x_over_eq}. Each part exhibits distinct solution behaviors, particularly in terms of the existence and uniqueness of solutions, which depend on the given inhomogeneities and the consistency conditions imposed on them.

Due to the block structure of the Kronecker canonical form \eqref{KCF_1},
the existence and uniqueness properties for the overall differential-algebraic equation \eqref{dae} can be deduced by analyzing the boundary value problems for each part of \eqref{111} individually.
Moreover, as the matrices $N_{\alpha}$, $L_{\beta}$, $K_{\beta}$, $L_{\gamma}^\top$, and $K_{\gamma}^\top$ have itself a block-diagonal structure, we will in a first step assume that each of these matrices consists of a single block, meaning that each of the multi-indices $\alpha$, $\beta$, and $\gamma$ consists of a single element.

\subsection{Existence and uniqueness theorem for the ODE part}
Consider the boundary value problem
\begin{equation}\label{x_difeq}\dot{x}_0(t)=A_0x_0(t)+f_0(t), \end{equation}
   \begin{equation}\label{bc_dif}
    B_0x_0(0)+C_0x_0(T)=d_0,\end{equation}
    where $A_0\in \mathbb{R}^{n_0\times n_0}$, $B_0, C_0\in\mathbb{R}^{l\times n_0}$, $d_0\in\mathbb{R}^l$.

    Let us define the matrix $Q_0\in\mathbb{R}^{l\times n_0}$ and the vector $\widehat{d}_0\in\mathbb{R}^l$ as follows:
    \begin{equation}\label{Q_0&d_0_hat}
        Q_0:=B_0+C_0 e^{A_0T}, \quad \widehat{d}_0:=d_0-C_0\int\limits_0^T e^{A_0(T-s)}f_0(s)ds.
    \end{equation}

\begin{theorem}\label{th_ODE_part}
   Let $f_0\in C([0,T],\mathbb{R}^{n_0})$. Then the
    problem \eqref{x_difeq}, \eqref{bc_dif}
has a solution $x_0\in C^1([0,T],\mathbb{R}^{n_0})$ if and only if \begin{equation}\label{d_0_orth}
    \widehat{d}_0\in \textnormal{ran}~ Q_0.
    \end{equation}

 In this case, any solution $x_0$ has a representation of the form
\begin{equation}\label{sol_x0_gen}
    x_0(t)=\int\limits_0^t e^{A_0(t-s)}f_0(s)ds+e^{A_0t}\left[Q_0^g\widehat{d}_0+(I_{n_0}-Q_0^g Q_0)z\right],
\end{equation} where  $z$ is an arbitrary vector in $\mathbb{R}^{n_0}$.
\end{theorem}

\begin{proof}
We apply the method of parameterization \cite{Dzhumabaev:1989} to problem \eqref{x_difeq}, \eqref{bc_dif}. By introducing a parameter $\lambda\in\mathbb{R}^{n_0}$ defined as
\begin{equation*}
    \lambda:=x_0(0),
\end{equation*}
and by substitution
\begin{equation*}
    u(t):=x_0(t)-\lambda,
\end{equation*}
Equation \eqref{x_difeq} is transformed into the initial value problem for the ordinary differential equation with parameter
\begin{equation*}\label{ivp_difpart}
\dot{u}(t)=A_0(u(t)+\lambda)+f_0(t),\quad
u(0)=0.
\end{equation*}

 For any fixed  $\lambda\in \mathbb{R}^{n_0}$ and any $f_0\in C([0,T],\mathbb{R}^{n_0})$, this problem has a unique solution
\begin{equation*}
\label{u_diff}
  u(t)=\int\limits_0^t e^{A_0(t-s)}(A_0\lambda+f_0(s)) ds.
  \end{equation*}

Hence,
\begin{equation*}
    x_0(t)=\lambda+u(t)=\left(I_{n_0}+\int\limits_0^t e^{A_0(t-s)}A_0ds\right)\lambda+\int\limits_0^t e^{A_0(t-s)}f_0(s)ds.
\end{equation*}

Taking into account the identity $I_{n_0}+\int\limits_0^t e^{A_0(t-s)} A_0ds=e^{A_0t}$, we obtain
\begin{equation}\label{x0_via_lambda}
    x_0(t)=e^{A_0t}\lambda+\int\limits_0^t e^{A_0(t-s)}f_0(s)ds.
\end{equation}

Substituting $x_0(0)$ and $x_0(T)$ into the boundary condition \eqref{bc_dif} yields that
the parameter $\lambda\in\mathbb{R}^{n_0}$ satisfies the system of algebraic equations
\begin{equation}\label{Q_syst_difpart}
    Q_0\lambda=\widehat{d}_0.
\end{equation}

If $\hat{d}_0\in \text{ran~}Q_0$, all solutions of \eqref{Q_syst_difpart} are given by
\begin{equation}\label{lambda_gen_difpart}
    \lambda=Q_0^g \widehat{d}_0+(I_{n_0}-Q_0^g Q_0)z,
\end{equation}
where  $z$ is an arbitrary vector in $\mathbb{R}^{n_0}$ and $Q_0^g$ is any generalized inverse of $Q_0$ (see \eqref{le} and \eqref{Solution_via_GeneralizedInverse}).  The representation \eqref{sol_x0_gen} for a solution of problem \eqref{x_difeq}, \eqref{bc_dif} is obtained by substituting  \eqref{lambda_gen_difpart}  into \eqref{x0_via_lambda}.
\end{proof}

\begin{corollary}\label{corollary_ODEpart_uniqueness}
  Let $f_0\in C([0,T],\mathbb{R}^{n_0})$. Then problem \eqref{x_difeq}, \eqref{bc_dif} is uniquely solvable if and only if, in addition to \eqref{d_0_orth},
\begin{equation}\label{rankQ_0}
    \emph{rank}~Q_0=n_0,
\end{equation}
in which case the solution is given by
\begin{equation}\label{sol_x0_unique}
    x_0(t)=\int\limits_0^t e^{A_0(t-s)}f_0(s)ds+e^{A_0t}(Q_0^\top Q_0)^{-1}Q_0^\top\widehat{d}_0.
\end{equation}
\end{corollary}
\begin{proof}
If condition \eqref{rankQ_0} is satisfied, then the matrix $Q_0$ has full column rank, implying that its kernel contains only the zero vector. Consequently, the system of equations  \eqref{Q_syst_difpart} can have at most one solution. Moreover, the fulfillment of condition \eqref{d_0_orth} ensures the existence of a solution.

Since $Q_0$ has full column rank, it admits a left inverse, denoted by $(Q_0)_L^{-1}$, which serves as a generalized inverse of  $Q_0$.
The unique solution is given by expression \eqref{sol_x0_unique}, which is derived from \eqref{sol_x0_gen} by substituting $Q_0^g$ with
$(Q_0)_L^{-1}=(Q_0^\top Q_0)^{-1}Q_0^\top$, see
\eqref{Alina}.
\end{proof}

\subsection{Existence and uniqueness theorem for the nilpotent part}
Consider the boundary value problem
\begin{equation}\label{x_alg_eq}
N_{\alpha}\dot{x}_1(t)=x_1(t)+f_1(t),
\end{equation}
   \begin{equation}\label{bc_alg}
    B_1x_1(0)+C_1x_1(T)=d_1,\end{equation}
    where $\alpha\in\mathbb{N}$, $N_{\alpha}\in\mathbb{R}^{\alpha\times\alpha}$, $B_1, C_1 \in\mathbb{R}^{l\times\alpha}$, $d_1\in\mathbb{R}^l$.
\begin{theorem}\label{th_nilpotent}
    Let $f_1\in C^{\alpha}([0,T],\mathbb{R}^{\alpha})$. Then the problem \eqref{x_alg_eq}, \eqref{bc_alg} has a unique solution $x_1\in C^1([0,T],\mathbb{R}^{\alpha})$ if and only if
    \begin{equation}\label{cond_x1}
        -\sum\limits_{i=0}^{\alpha-1}\left[B_1N_{\alpha}^i f_1^{(i)}(0)+C_1N_{\alpha}^i f_1^{(i)}(T)\right]=d_1.
    \end{equation}
   The solution is given by \begin{equation}\label{sol_x1}
        x_1(t)=-\sum\limits_{i=0}^{\alpha-1}N_{\alpha}^i f_1^{(i)}(t).
    \end{equation}
\end{theorem}
\begin{proof} Equation \eqref{x_alg_eq} has a unique solution, which is given by \eqref{sol_x1} (see, e.g., \cite[Lemma 2.8]{Kunkel:2006}). It follows immediately that problem \eqref{x_alg_eq}, \eqref{bc_alg} is solvable if and only if $x_1(t)$,  when substituted into the boundary condition \eqref{bc_alg}, satisfies the identity \eqref{cond_x1}.
\end{proof}

\subsection{Existence theorem for the underdetermined part}\label{Subsection_UNDERDET_PART}
Consider the boundary value problem \begin{equation}\label{x_under_eq}
L_{\beta}\dot{x}_2 (t)=K_{\beta}x_2(t)+
f_2(t),
\end{equation}
\begin{equation}\label{bc_under}
    B_2x_2(0)+C_2x_2(T)=d_2,\end{equation}
where $\beta\in\mathbb{N}$, $B_2, C_2 \in\mathbb{R}^{l\times\beta}$, $d_2\in\mathbb{R}^l$.

If $\beta=1$ in \eqref{x_under_eq}, then $x_2$ is not constrained by any equation and can therefore be chosen arbitrarily. The only condition for the solvability of problem \eqref{x_under_eq}, \eqref{bc_under} in this case is that $x_2$ satisfies the boundary condition \eqref{bc_under} with $B_2, C_2, d_2\in \mathbb{R}$.

Let us suppose $\beta>1$ and define
    \begin{equation}\label{Q2}
Q_2:=\left(B_2+C_2e^{N_{\beta}T}\right)L_{\beta}^\top,
    \end{equation}
    \begin{equation}\label{h_beta}
h_{\beta}(t; \varphi(t)):=e_{1,\beta}\varphi(t)+\int\limits_0^t e^{N_{\beta}(t-s)}L_{\beta}^\top\left[f_2(s)+e_{1,\beta-1}\varphi(s)\right]ds,
    \end{equation}
    \begin{equation}\label{d2}
 \widehat{d}_2:=d_2-B_2e_{1,\beta}\varphi(0)-C_2 h_{\beta}(T;  \varphi(T)).
    \end{equation}
  Here  $e_{1,\beta}=(1,0,\ldots,0)^\top\in \mathbb{R}^{\beta}$, $e_{1,\beta-1}=(1,0,\ldots,0)^\top\in \mathbb{R}^{\beta-1}$, and $\varphi:[0,T]\rightarrow \mathbb{R}$ is an arbitrary continuous function.
\begin{theorem}\label{th_under_part}
   Let $\beta>1$ and $f_2\in C([0,T],\mathbb{R}^{\beta-1})$. Then, for arbitrary $\varphi\in C([0,T],\mathbb{R})$, the problem \eqref{x_under_eq}, \eqref{bc_under} has a solution $x_2\in C^1([0,T],\mathbb{R}^{\beta})$ if and only if \begin{equation*}\label{d_2_orth}
    \widehat{d}_2\in \textnormal{ran}~ Q_2.
    \end{equation*}

 In this case, any solution $x_2(t)$ has a representation
\begin{equation}\label{sol_x2_gen}
    x_2(t)=h_{\beta}(t; \varphi(t))+ e^{N_{\beta}t}L_{\beta}^\top\left[Q_2^g\widehat{d}_2+(I_{\beta-1}-Q_2^g Q_2)z\right],
\end{equation} where  $z$ is an arbitrary vector in $\mathbb{R}^{\beta-1}$.
\end{theorem}
\begin{proof}
    Let us rewrite equation \eqref{x_under_eq} componentwise as follows:
\begin{gather*}\label{underdet_component}
\dot{x}_{2,2}(t)=x_{2,1}(t)+f_{2,1}(t),\notag\\
\vdots \\
\dot{x}_{2,\beta}(t)=x_{2,\beta-1}(t)+f_{2,\beta-1}(t).\notag
\end{gather*}
Here we have $\beta-1$ equations for $\beta$ unknown functions $x_{2,1}, \ldots,  x_{2,\beta}$.
Following the procedure in \cite{berger_diss}, we write \eqref{x_under_eq} in a different way. We drop the first component of the vector-valued 
function~$x_2$ and denote the resulting by $x_-$,
\begin{equation*}
x_-(t):=(x_{2,2}(t), \ldots, x_{2,\beta}(t))^\top,
\end{equation*}
we can rewrite equation \eqref{x_under_eq} as
\begin{align}
    \label{x-}
\dot{x}_-(t)=N_{\beta-1}x_-(t)+ f_2(t)+ e_{1,\beta-1}x_{2,1}(t).
\end{align}
Hence, a solution exists for all $f_2 \in C([0,T],\mathbb{R}^{\beta-1})$ and all $x_{2,1}\in C([0,T], \mathbb{R})$. 

We select $x_{2,1}=\varphi$, where $\varphi$ is an arbitrary continuous function on the interval $[0,T]$. We introduce a parameter
\begin{equation*}
    \mu:=x_-(0) \quad \mbox{and set} \quad
    v(t):=x_-(t)-\mu.
\end{equation*}
Then \eqref{x-} is transformed into the initial value problem
\begin{equation*}
    \dot{v}(t)=N_{\beta-1}(v(t)+\mu)+ f_2(t)+ e_{1,\beta-1}\varphi(t),\quad v(0)=0.
\end{equation*}
For fixed $\mu$, this problem has a solution of the form
\begin{equation*}\label{under_sol}
    v(t)=\int\limits_0^t e^{N_{\beta-1}(t-s)}\left(N_{\beta-1}\mu+f_2(s) + e_{1,\beta-1}\varphi(s)\right)ds.
\end{equation*}
Hence, we obtain
\begin{equation*}\begin{aligned}
        x_-(t)&=\mu+v(t)\\
        &=\left(I_{\beta-1}+\int\limits_0^t e^{N_{\beta-1}(t-s)}N_{\beta-1}ds\right)\mu+\int\limits_0^t e^{N_{\beta-1}(t-s)}[f_2(s) + e_{1,\beta-1}\varphi(s)]ds\\
        &=e^{N_{\beta-1}t}\mu+\int\limits_0^t e^{N_{\beta-1}(t-s)}[f_2(s) + e_{1,\beta-1}\varphi(s)]ds
\end{aligned}
\end{equation*}
and
\begin{equation*}
\begin{aligned}
    x_2(t)&=\begin{bmatrix}
        \varphi(t)\\x_-(t)
\end{bmatrix}\\
&=e_{1,\beta}\varphi(t)+L_{\beta}^\top\left(e^{N_{\beta-1}t}\mu+\int\limits_0^t e^{N_{\beta-1}(t-s)}\left[f_2(s) + e_{1,\beta-1}\varphi(s)\right]ds\right).
\end{aligned}
\end{equation*}
Taking into account 
$L_{\beta}^\top N_{\beta-1}=N_{\beta}L_{\beta}^\top$
and, hence, 
$L_{\beta}^\top e^{N_{\beta-1}}=e^{N_{\beta}}L_{\beta}^\top$,  we obtain with  \eqref{h_beta}
\begin{equation}\label{x2_gen}\begin{aligned}
    x_2(t)=e^{N_{\beta}t}L_{\beta}^\top \mu+h_{\beta}(t;\varphi(t)).
\end{aligned}\end{equation}

We substitute the expressions for $x_2(0)$ and $x_2(T)$, derived from \eqref{x2_gen}, into the boundary condition \eqref{bc_under}. This substitution leads to the system of algebraic equations in the parameter $\mu\in\mathbb{R}^{\beta-1}$ 
\begin{equation}\label{Q2mu=d2}
   Q_2\mu=\widehat{d}_2,
\end{equation}
where $Q_2$ and $\widehat{d}_2$ are defined by \eqref{Q2} and \eqref{d2}, respectively.

Equation \eqref{Q2mu=d2}  has a solution if and only if $\widehat{d}_2 \in \mbox{ran}\, Q_2$. In this case, all solutions of \eqref{Q2mu=d2} are given by (see \eqref{le} and \eqref{Solution_via_GeneralizedInverse})
\begin{equation*}
    \mu=Q_2^g\widehat{d}_2+(I_{\beta-1}-Q_2^g Q_2)z,
\end{equation*}
where  $z$ is an arbitrary vector in $\mathbb{R}^{\beta-1}$.

The general solution of the problem \eqref{x_under_eq}, \eqref{bc_under} is obtained by substituting this expression for $\mu$ into \eqref{x2_gen}, yielding \eqref{sol_x2_gen}.
\end{proof}

\subsection{Existence and uniqueness theorem for the overdetermined part}

Consider the boundary value problem
\begin{equation}\label{x_over_eq}
L_{\gamma}^\top\dot{x}_3(t)=K_{\gamma}^\top x_3(t)+f_3(t),
\end{equation}
\begin{equation}\label{bc_over}
    B_3x_3(0)+C_3x_3(T)=d_3,\end{equation}
where $\gamma\in\mathbb{N}$, $B_3, C_3 \in\mathbb{R}^{l\times(\gamma-1)}$, $d_3\in\mathbb{R}^l$.

If $\gamma=1$, then the variable $x_3$ does not appear in equation \eqref{x_over_eq}, which simplifies to $f_3=0$. Moreover, no boundary condition is imposed. 

\begin{theorem}\label{th_over_BVP}
    Let $\gamma>1$ and $f_3\in C^{\gamma}([0,T],\mathbb{R}^{\gamma})$. Then the problem \eqref{x_over_eq}, \eqref{bc_over} has a unique solution $x_3\in C^1([0,T],\mathbb{R}^{\gamma-1})$ if and only if
    \begin{itemize}
        \item[\emph{(a)}] $\sum\limits_{i=1}^{\gamma}f_{3,i}^{(\gamma-i)}(t)=0$;
         \item[\emph{(b)}]
         $-B_3 K_{\gamma}\sum\limits_{i=0}^{\gamma-1}N_{\gamma}^i f_3^{(i)}(0)-C_3K_{\gamma}\sum\limits_{i=0}^{\gamma-1}N_{\gamma}^i f_3^{(i)}(T)=d_3.$
        \end{itemize}

      The solution is given by \begin{equation}\label{sol_x3}
        x_3(t)=-K_{\gamma}\sum\limits_{i=0}^{\gamma-1}N_{\gamma}^i f_3^{(i)}(t).
    \end{equation}
\end{theorem}
\begin{proof}
    System \eqref{x_over_eq}, due to the structure of the matrices $L_{\gamma}^T$ and $K_{\gamma}^T$, is overdetermined. Indeed, let us consider its componentwise version:
\begin{align}\label{x_overdet_component}
0&=x_{3,1}(t)+f_{3,1}(t),\notag\\
\dot{x}_{3,1}(t)&=x_{3,2}(t)+f_{3,2}(t),\notag\\
&\vdots \\
\dot{x}_{3,\gamma-2}(t)&=x_{3,\gamma-1}(t)+f_{3,\gamma-1}(t),\notag\\
\dot{x}_{3,\gamma-1}(t)&=f_{3,\gamma}(t).\notag
\end{align}
So we have $\gamma$  equations for $\gamma-1$ unknown functions $x_{3,1},\ldots$, $x_{3,\gamma-1}$.
Following the procedure outlined in \cite{berger_diss}, we introduce the vector 
$$
x_+(t):=(x_{3,1}(t), \ldots, x_{3,\gamma-1}(t), 0)
$$
and represent system \eqref{x_overdet_component} in the equivalent form
\begin{equation*}
N_{\gamma}\dot{x}_+(t)=x_+(t)+f_3(t).
\end{equation*}
The solution of the above equation is
\begin{equation}\label{x+sol}
    x_+(t)=-\sum\limits_{i=0}^{\gamma-1} N_{\gamma}^i f_3^{(i)}(t).
\end{equation}
The first $\gamma-1$ components of the vector-valued function $x_+$ in \eqref{x+sol} constitute the solution of equation \eqref{x_over_eq}, which is given by \eqref{sol_x3}.
In addition, the last component of $x_+(t)$ in \eqref{x+sol} imposes  condition (a).

As we can see, the solution of equation \eqref{x_over_eq} is uniquely defined by expression \eqref{sol_x3}. This implies that problem \eqref{x_over_eq}, \eqref{bc_over} is solvable if and only if
$x_3$ from \eqref{sol_x3} satisfies the boundary condition \eqref{bc_over}.  The substitution of the values $x_3(0)$ and $x_3(T)$ leads directly to condition (b).
\end{proof}

\section{Existence and uniqueness theorems for the general boundary value problem}

In this section, we return to the general setting of the boundary value problem for the differential-algebraic equation \eqref{dae}, \eqref{BC}, building on the results from the previous section. Specifically, we relax the earlier assumption that the multiindices \( \alpha \), \( \beta \), and \( \gamma \) in the decomposition \eqref{111} consist of single elements only. We now consider
\[
\alpha \in \mathbb{N}^{n_{\alpha}}, \quad \beta \in \mathbb{N}^{n_{\beta}}, \quad \gamma \in \mathbb{N}^{n_{\gamma}}.
\]

We make the following assumptions on the components of $Wf$:
\begin{itemize}
    \item[\textbf{(A0)}] \( f_0 \in C([0,T], \mathbb{R}^{n_0}) \);
    \item[\textbf{(A1)}] \( f_1 = \left(f_{\alpha_1}^\top, \ldots, f_{\alpha_{n_{\alpha}}}^\top\right)^\top \in C^{\alpha_1}([0,T], \mathbb{R}^{\alpha_1}) \times \ldots \times C^{\alpha_{n_{\alpha}}}([0,T], \mathbb{R}^{\alpha_{n_{\alpha}}}) \);
    \item[\textbf{(A2)}] \( f_2 \in C([0,T], \mathbb{R}^{|\beta| - n_{\beta}}) \);
    \item[\textbf{(A3)}] \( f_3 = \left(f_{\gamma_1}^\top, \ldots, f_{\gamma_{n_{\gamma}}}^\top\right)^\top \in C^{\gamma_1}([0,T], \mathbb{R}^{\gamma_1}) \times \ldots \times C^{\gamma_{n_{\gamma}}}([0,T], \mathbb{R}^{\gamma_{n_{\gamma}}}) \).
\end{itemize}

As observed in the previous section, the solvability criteria for the ODE part and the underdetermined part with \( \beta_k > 1 \) are formulated in terms of certain matrices \( Q_0 \) and \( Q_2 \). In what follows, we distinguish the following two cases:
$ n_0 + |\beta| - n_{\beta} > 0$ (case I) and $ n_0 + |\beta| - n_{\beta} = 0$ (case II). 
As, by definition,  $|\beta|\geq n_\beta$, the quantity $ n_0 + |\beta| - n_{\beta}$
is always nonnegative.

\begin{itemize}
    \item[I.] \( n_0 + |\beta| - n_{\beta} > 0 \): that is, at least one of the blocks corresponding to the ODE part or the underdetermined part is present in the decomposition \eqref{KCF_1}; and if the ODE part is absent, then not all components of \( \beta \) are equal to 1.
    
    \item[II.] \( n_0 + |\beta| - n_{\beta} = 0 \): that is, 
    $n_0=0$ and $n_{\beta} =|\beta|$.
    The ODE part is absent in \eqref{KCF_1}, and the underdetermined part is either absent or consists entirely of components equal to 1.
\end{itemize}

\subsection{Case I: \( n_0 + |\beta| - n_{\beta} > 0 \)}

We set
\begin{equation*}
    M_{\beta}:=\text{diag}\left(e_{\beta_1,1},\ldots, e_{\beta_{n_{\beta}},1}\right)\in\mathbb{R}^{|\beta|\times n_{\beta}},
\end{equation*}
where $e_{\beta_k,1}=(1,0,\ldots,0)^\top\in\mathbb{R}^{\beta_k}$, $k=1,\ldots,n_{\beta}$,
are the first unit vectors. We also set $\overline{\beta}:=(\beta_1-1,\ldots,\beta_{n_{\beta}}-1)$, and
\begin{equation*}
    M_{\overline{\beta}}:=\text{diag}\left(e_{\beta_1-1,1},\ldots, e_{\beta_{n_{\beta}-1},1}\right)\in\mathbb{R}^{|\overline{\beta}|\times n_{\beta}}.
\end{equation*}
If some of the entries $\beta_i$ of $\beta$ are equal to 1, then $ M_{\overline{\beta}}$ contains a zero column vector at the corresponding positions, meaning that the row containing $e_{\beta_i,1}$ in $M_{\beta}$ is cancelled. Let, for instance, $\beta=(3,1,2)$. Then $\overline{\beta}=(2,0,1)$,
\begin{equation*}
    M_{\beta}=\begin{bmatrix}
        1&0&0\\0&0&0\\0&0&0\\0&1&0\\0&0&1\\0&0&0
    \end{bmatrix}~~\text{and} ~~M_{\overline{\beta}}=\begin{bmatrix}
        1&0&0\\0&0&0\\0&0&1
    \end{bmatrix}.
\end{equation*}

Let $\widetilde{\varphi}\in C([0,T],\mathbb{R}^{n_{\beta}})$ be an arbitrary function. If $|\beta|>n_\beta$, we define
\begin{equation*}
\label{h_beta_GEN}
\widetilde{h}_{\beta}(t; \widetilde{\varphi}(t)):=M_{\beta}\widetilde{\varphi}(t)+\int\limits_0^t e^{N_{\beta}(t-s)}L_{\beta}^\top\left[f_2(s)+M_{\bar{\beta}}\widetilde{\varphi}(s)\right]ds
.
\end{equation*}

We now construct a matrix \( Q \in \mathbb{R}^{l \times (n_0 + |\beta| - n_{\beta})} \) and a vector \( \widehat{d} \in \mathbb{R}^{l} \) as follows:

\begin{equation}\label{Q}
    Q :=      \begin{cases}
        \begin{bmatrix}
            B_0 + C_0 e^{A_0 T}, \; (B_2 + C_2 e^{N_{\beta} T}) L_\beta^\top
        \end{bmatrix}, &\text{if   } |\beta|>n_\beta,\\
        ~B_0 + C_0 e^{A_0 T},&\text{if   } |\beta|=n_\beta,
    \end{cases}   
\end{equation}

\begin{equation}\label{d_hat}
    \widehat{d} := d - \sum\limits_{i=0}^3 \bar{d}_i,
\end{equation}
where
\begin{equation*}
    \bar{d}_0:=C_0\int\limits_0^T e^{A_0(T-s)}f_0(s)ds,\qquad 
    \bar{d}_1:=-\sum\limits_{i=0}^{\nu_{\alpha}-1}\left(B_1 N_{\alpha}^i f_1^{(i)}(0)+C_1N_{\alpha}^i f_1^{(i)}(T)\right),
\end{equation*}
\begin{equation*}
    \bar{d}_2:=\begin{cases}B_2M_{\beta}\widetilde{\varphi}(0)+C_2 \widetilde{h}_{\beta}(T; \widetilde{\varphi}(T)),&\text{if   }|\beta|>n_\beta,\\
    B_2\widetilde{\varphi}(0)+C_2  \widetilde{\varphi}(T),&\text{if   }|\beta|=n_\beta,
    \end{cases}
\end{equation*}
\begin{equation*}
\bar{d}_3:=\begin{cases}
    -\sum\limits_{i=0}^{\nu_{\gamma}-1}\left(B_3K_{\gamma}N_{\gamma}^i f_3^{(i)}(0)+C_3K_{\gamma}N_{\gamma}^i f_3^{(i)}(T)\right),&\text{if   }|\gamma|>n_\gamma,\\
    ~0, & \text{if   } |\gamma|=n_\gamma.
\end{cases}
\end{equation*}
Here $\nu_{\alpha}:=\max \{\left.\alpha_k\right\vert k=1,\ldots,n_{\alpha}\}$, and $    \nu_{\gamma}:=\max \{\left.\gamma_k\right\vert k=1,\ldots,n_{\gamma}\}$.

We set \[
P_0 := \begin{bmatrix} I_{n_0} & O_{n_0 \times (|\beta| - n_{\beta})} \end{bmatrix}, \quad
P_2 := \begin{bmatrix} O_{(|\beta| - n_{\beta}) \times n_0} & I_{|\beta| - n_{\beta}} \end{bmatrix}.
\]

In the following theorem we characterize the solvability of the problem \eqref{dae}, \eqref{BC}. Together with Theorem \ref{Th1_second_case}, it is the main result.

\begin{theorem}\label{Th1}
Let \( W \) and \( V \) be invertible matrices that transform the matrix pair \( (E, A) \) in \eqref{dae} into the quasi-Kronecker canonical form \eqref{KCF_1} with
\(n_0 + |\beta| - n_{\beta} > 0\). Suppose that assumptions \textnormal{\textbf{(A0)}--\textbf{(A3)}} are satisfied.

Then, for any \( \widetilde{\varphi} \in C([0,T],\mathbb{R}^{n_{\beta}}) \), the boundary value problem \eqref{dae}, \eqref{BC} has a solution \( x \in C^1([0,T],\mathbb{R}^n) \) if and only if the following conditions hold:
\begin{itemize}
    \item[\emph{(a)}] \(
        \sum\limits_{i=1}^{\gamma_k} f_{\gamma_k,i}^{(\gamma_k - i)}(t) = 0,\quad  k = 1,\ldots,n_{\gamma};
    \)
    \item[\emph{(b)}] \( \widehat{d} \in \operatorname{ran} Q \).
\end{itemize}

Any solution \( x \) is given by
\begin{equation}\label{SOLUTION}
    x(t) = V
\begin{bmatrix}
x_0(t) \\
x_1(t) \\
x_2(t) \\
x_3(t)
\end{bmatrix},
\end{equation}
where
\begin{align*}
x_0(t) &= e^{A_0 t} P_0 \left(Q^g \widehat{d} + \big(I_{n_0 + |\beta| - n_{\beta}} - Q^g Q\big) z\right) 
          + \int_0^t e^{A_0(t - s)} f_0(s)\, ds, \\
x_1(t) &= -\sum\limits_{i=0}^{\nu_{\alpha} - 1} N_{\alpha}^i f_1^{(i)}(t), \\
x_2(t) &=\begin{cases}
   e^{N_{\beta} t} L_{\beta}^\top P_2 \left(Q^g \widehat{d} + \big(I_{n_0 + |\beta| - n_{\beta}} - Q^g Q\big) z\right) + \widetilde{h}_{\beta}(t;\widetilde{\varphi}(t)), & \text{if } |\beta|-n_{\beta}>0,\\
   \widetilde{\varphi}(t), & \text{if } |\beta|-n_{\beta}=0,
\end{cases} \\
x_3(t) &= -K_{\gamma} \sum\limits_{i=0}^{\nu_{\gamma} - 1} N_{\gamma}^i f_3^{(i)}(t), \quad \text{if } |\gamma| > n_\gamma.
\end{align*}
Here \( z \in \mathbb{R}^{n_0 + |\beta| - n_{\beta}} \) is an arbitrary vector 
and \(Q^g\) is any generalized inverse of \(Q\).
\smallskip

If \( |\gamma| = n_{\gamma} \), then \( x_3(t) \) is absent from the representation \eqref{SOLUTION}, and the solution consists only of the first three components:
\[
x(t) = V
\begin{bmatrix}
x_0(t) \\
x_1(t) \\
x_2(t)
\end{bmatrix}.
\]
\end{theorem}

\begin{proof}
The transformation $x(t) = V\left(x_0^\top(t), x_1^\top(t), x_2^\top(t), x_3^\top(t)\right)^\top$ and left multiplication of equation \eqref{dae} by $W$ yield the decoupled system \eqref{111}. The boundary condition \eqref{BC} is transformed accordingly, resulting in the component-wise form \eqref{BC_decomposed}.

\begin{itemize}
    \item[(i)] \emph{The ODE part.} The subsystem for $x_0$ is an ordinary differential equation with boundary condition:
    \[
    \dot{x}_0(t) = A_0 x_0(t) + f_0(t), \quad B_0 x_0(0) + C_0 x_0(T) = d_0.
    \]
    We construct the matrix $Q_0 = B_0 + C_0 e^{A_0 T}$. By Theorem~\ref{th_ODE_part}, under assumption \textbf{(A0)}, this problem has a solution $x_0 \in C^1([0,T],\mathbb{R}^{n_0})$ if and only if
    \[
    \widehat{d}_0 := d_0 - \bar{d}_0  \in \text{ran } Q_0.
    \]
    In this case, the general solution is given by \eqref{sol_x0_gen}.

\item[(ii)] \emph{The nilpotent part.} For $\alpha \in \mathbb{N}^{n_{\alpha}}$, we consider the subsystem
\begin{equation}\label{alg_BVP_gen_proof}
    N_{\alpha} \dot{x}_1(t) = x_1(t) + f_1(t), \quad B_1 x_1(0) + C_1 x_1(T) = d_1.
\end{equation}
Due to the block-diagonal structure of $N_{\alpha}$, the functions $x_1$ and $f_1$ decompose into components $x_{\alpha_k}$ and $f_{\alpha_k}$ for $k=1,2,\ldots,n_\alpha$. Correspondingly, the boundary matrices and vector decompose as
\[
    B_1 = \big[B_{\alpha_1}, \ldots, B_{\alpha_{n_\alpha}}\big], \quad 
    C_1 = \big[C_{\alpha_1}, \ldots, C_{\alpha_{n_\alpha}}\big], \quad 
    d_1 = \sum_{k=1}^{n_\alpha} d_{\alpha_k}.
\]
Thus, the problem \eqref{alg_BVP_gen_proof} splits into $n_\alpha$ independent subproblems:
\[
    N_{\alpha_k} \dot{x}_{\alpha_k}(t) = x_{\alpha_k}(t) + f_{\alpha_k}(t), \quad 
    B_{\alpha_k} x_{\alpha_k}(0) + C_{\alpha_k} x_{\alpha_k}(T) = d_{\alpha_k}, 
\]
for $k = 1,\ldots,n_\alpha$. By Theorem~\ref{th_nilpotent}, under assumption \textbf{(A1)}, each subproblem admits a solution $x \in C^1([0,T],\mathbb{R}^\alpha)$ if and only if
\[
    \widehat{d}_{\alpha_k}:=d_{\alpha_k}+\sum_{i=0}^{\alpha_k - 1} \left[ B_{\alpha_k} N_{\alpha_k}^i f_{\alpha_k}^{(i)}(0) + C_{\alpha_k} N_{\alpha_k}^i f_{\alpha_k}^{(i)}(T) \right] = 0
\]
for $k = 1,\ldots,n_\alpha$.
In this case, the unique solution is given by
\[
    x_{\alpha_k}(t) = -\sum_{i=0}^{\alpha_k - 1} N_{\alpha_k}^i f_{\alpha_k}^{(i)}(t), \quad k = 1,\ldots,n_\alpha.
\]

Consequently, the solution of the full system \eqref{alg_BVP_gen_proof} is
\begin{equation}\label{sol_x1_genBVP}
    x_1(t) = -\sum_{i=0}^{\nu_\alpha - 1} N_{\alpha}^i f_1^{(i)}(t),
\end{equation}
and it exists if and only if
\begin{equation*}\label{cond_x1_alpha}
    \widehat{d}_1:=d_1-\bar{d}_1= 0.
\end{equation*}

\item[(iii)] \emph{The underdetermined part.}  
For \( \beta \in \mathbb{N}^{n_{\beta}} \), consider the boundary value problem
\begin{equation}\label{under_BVP_gen_proof}
   L_\beta \dot{x}_2(t) = K_\beta x_2(t) + f_2(t), \quad B_2 x_2(0) + C_2 x_2(T) = d_2,
\end{equation}
which, analogously to the previous case, decomposes into subproblems
\[\begin{aligned}
    L_{\beta_k} \dot{x}_{\beta_k}(t) = K_{\beta_k} x_{\beta_k}(t) +& f_{\beta_k}(t), \quad 
B_{\beta_k} x_{\beta_k}(0) + C_{\beta_k} x_{\beta_k}(T) = d_{\beta_k}, \\
&k = 1, \ldots, n_\beta.
\end{aligned}
\]
Let \( \varphi_k(t) \in C([0,T], \mathbb{R}) \) be an arbitrary function.

For subproblems with \( \beta_k > 1 \), we apply Theorem~\ref{th_under_part}. We define the matrices \( Q_{\beta_k} \), the functions \( h_{\beta_k} \), and the vectors \( \widehat{d}_{\beta_k} \) as in \eqref{Q2}, \eqref{h_beta}, and \eqref{d2}, respectively, 
where one replaces $\beta$ by $\beta_k$, $\varphi$ by $\varphi_k$,
$f_2$ by $f_{\beta_k}$, $B_2$ by $B_{\beta_k}$, $C_2$ by $C_{\beta_k}$, and $d_2$ by $d_{\beta_k}$. Then, under assumption \textbf{(A2)}, the subproblem admits a solution if and only if
\[
\widehat{d}_{\beta_k} \in \operatorname{ran} Q_{\beta_k}.
\]
In this case, any solution $x_{\beta_k} \in C^1([0,T],\mathbb{R}^{\beta_k})$ is given by
\begin{equation*}\label{sol_xbeta_k_proof}
x_{\beta_k}(t) = h_{\beta_k}(t; \varphi_k(t)) + e^{N_{\beta_k} t} L_{\beta_k}^\top \left[ Q_{\beta_k}^g \widehat{d}_{\beta_k} + \left(I_{\beta_k - 1} - Q_{\beta_k}^g Q_{\beta_k} \right) z_k \right],
\end{equation*}
where \( z_k \in \mathbb{R}^{\beta_k - 1} \) is an arbitrary vector.

In the case $\beta_k = 1$, the function $x_{\beta_k}(t)$ is not constrained by the differential equation and may be chosen freely: $x_{\beta_k}(t) = \varphi_k(t)$, subject only to the boundary condition
\begin{equation*}\label{BC_condition_beta1}
   \hat{\hat{d}}_{\beta_k} := d_{\beta_k} - B_{\beta_k} \varphi_k(0) - C_{\beta_k} \varphi_k(T) = 0.
\end{equation*}

To formulate the result for the full system \eqref{under_BVP_gen_proof}, we first assume that not all components satisfy $\beta_k=1$, i.e.\ $|\beta|-n_\beta>0$.  Let \( \beta_{k_1}, \ldots, \beta_{k_r} \) be the components of \( \beta \) such that \( \beta_{k_i} > 1 \), listed in the same order as they appear in \( \beta \). We define the block matrix
\[
Q_2 := \begin{bmatrix}
Q_{\beta_{k_1}}, \ldots,  Q_{\beta_{k_r}}
\end{bmatrix},
\]
where each \( Q_{\beta_{k_i}} \) is defined as in \eqref{Q2}:
\[
Q_{\beta_{k_i}} = \left( B_{\beta_{k_i}} + C_{\beta_{k_i}} e^{N_{\beta_{k_i}} T} \right) L_{\beta_{k_i}}^\top.
\]

Then all subproblems of \eqref{under_BVP_gen_proof} corresponding to indices with \(\beta_k>1\) are simultaneously solvable if and only if
\[
\sum_{i=1}^{r} \widehat{d}_{\beta_{k_i}}
= \sum_{i=1}^{r} \bigl( d_{\beta_{k_i}} - B_{\beta_{k_i}} e_{1,\beta_{k_i}} \varphi_{k_i}(0) - C_{\beta_{k_i}} h_{\beta_{k_i}}(T;\,\varphi_{k_i}(T)) \bigr) \in \operatorname{ran} Q_2.
\]

The remaining subproblems of \eqref{under_BVP_gen_proof} corresponding to indices with $\beta_k=1$ are simultaneously solvable if and only if the sum of the associated  $\hat{\hat{d}}_{\beta_k}$ is zero. We then define the vector
\begin{equation*}\label{hat_d_2}
\begin{aligned}
    \widehat{d}_2 
&:= \sum_{i=1}^{r} \widehat{d}_{\beta_{k_i}} +
 \sum_{j=1}^{n_\beta-r} \hat{\hat{d}}_{\beta_{k_j}}\\
&=\sum_{k=1}^{n_\beta} \big[ d_{\beta_k} - B_{\beta_k} e_{1,\beta_k} \varphi_k(0) - C_{\beta_k} h_{\beta_k}(T; \varphi_k(T)) \big]= d_2 - \bar{d}_2.
\end{aligned}
\end{equation*}
Hence, under assumption \textbf{(A2)}, the boundary value problem \eqref{under_BVP_gen_proof} admits a solution if and only if
\[
\widehat{d}_2 \in  \operatorname{ran} Q_2.
\]

It follows from the definition of $M_\beta$ and \eqref{h_beta}
that
any solution has the representation
\begin{equation}\label{sol_xbeta_proof}
    x_2(t) = \widetilde{h}_\beta(t; \widetilde{\varphi}(t)) + e^{N_\beta t} L_\beta^\top \left[ Q_2^g \widehat{d}_2 + (I_{|\beta| - n_\beta} - Q_2^g Q_2) z \right],
\end{equation}
where $z\in\mathbb{R}^{|\beta| - n_\beta}$ is an arbitrary vector and
 $\widetilde{\varphi}\in C([0,T],\mathbb{R}^{n_{\beta}})$ be an arbitrary function. 
\smallskip

\noindent
In the case when all $\beta_k=1$, i.e., $|\beta|=n_\beta$, 
the solution of problem \eqref{under_BVP_gen_proof} consists solely of arbitrary continuous components:
\begin{equation*}
    x_2(t)=\left(\varphi_1(t),\ldots,\varphi_{n_\beta}(t)\right)^\top=\widetilde{\varphi}(t),
\end{equation*}
which must only satisfy the boundary condition in \eqref{under_BVP_gen_proof}:
\begin{equation*}
    \widehat{d}_2:=d_2-B_2\widetilde{\varphi}(0)-C_2\widetilde{\varphi}(T)=0.
\end{equation*}

\item[(iv)] \emph{The overdetermined part.} For $\gamma \in \mathbb{N}^{n_{\gamma}}$, consider the subsystem
\begin{equation}\label{over_BVP_gen_proof}
L_{\gamma}^\top\dot{x}_3(t)=K_{\gamma}^\top x_3(t)+f_3(t),
\quad    B_3x_3(0)+C_3x_3(T)=d_3.
\end{equation}
Due to the block structure of the overdetermined part, the functions $x_3$ and $f_3$ decompose into components $x_{\gamma_k}$ and $f_{\gamma_k}$ for $k = 1, 2, \ldots, n_\gamma$, and the boundary condition matrices decompose accordingly. 
This leads to subproblems of the form
\begin{equation}\label{subproblem_over_proof}
    L_{\gamma_k}^\top \dot{x}_{\gamma_k}(t)=K_{\gamma_k}^\top x_{\gamma_k}(t) + f_{\gamma_k}(t) , \quad B_{\gamma_k} x_{\gamma_k}(0) + C_{\gamma_k} x_{\gamma_k}(T) = d_{\gamma_k}, 
\end{equation}
for $k = 1, \ldots, n_\gamma$. To subproblems with $\gamma_k > 1$ we apply Theorem~\ref{th_over_BVP}. Then, under assumption \textbf{(A3)}, the subproblem \eqref{subproblem_over_proof} has a solution if and only if the following conditions hold:
\begin{itemize}
        \item[($\text{a}'$)] $\sum\limits_{i=1}^{\gamma_k}f_{\gamma_k,i}^{(\gamma_k-i)}(t)=0$;
         \item[($\text{b}'$)]
    $\widehat{d}_{\gamma_k}:=d_{\gamma_k}+B_{\gamma_k} K_{\gamma_k}\sum\limits_{i=0}^{\gamma_k-1}N_{\gamma_k}^i f_{\gamma_k}^{(i)}(0)+C_{\gamma_k}K_{\gamma_k}\sum\limits_{i=0}^{\gamma_k-1}N_{\gamma_k}^i f_{\gamma_k}^{(i)}(T)=0.$
        \end{itemize}

In this case, the solution is uniquely given by
\[
    x_{\gamma_k}(t) = -K_{\gamma_k}\sum\limits_{i=0}^{\gamma_k-1}N_{\gamma_k}^i f_{\gamma_k}^{(i)}(t).
\]

For the components with $\gamma_k = 1$, the differential-algebraic equation reduces to the algebraic condition
\begin{equation*}\label{f_gamma1_cond}
    f_{\gamma_k}(t) = 0,
\end{equation*}
and there are no boundary conditions.

Let us now  aggregate the solutions across all components. If not all components of $\gamma$ are equal to 1, i.e., $|\gamma|>n_\gamma$, then the solution of the full subsystem \eqref{over_BVP_gen_proof} is given by
\begin{equation}\label{sol_x3_genBVP}
    x_3(t) = -K_{\gamma}\sum\limits_{i=0}^{\nu_\gamma-1}N_{\gamma}^i f_3^{(i)}(t),
\end{equation}
and this solution exists if and only if the following conditions hold:
\begin{itemize}
        \item[($\text{a}''$)] $\sum\limits_{i=1}^{\gamma_k}f_{\gamma_k,i}^{(\gamma_k-i)}(t)=0$, ~~~$k=1,\ldots,n_{\gamma}$;
         \item[($\text{b}''$)] $\widehat{d}_3:=d_3-\bar{d}_3=0.$
\end{itemize}
        
        In the case \( \gamma = (1, 1, \ldots, 1) \), i.e., \( |\gamma| = n_\gamma \), the system reduces to a set of algebraic constraints that must be satisfied:
\begin{equation*}
    f_{\gamma_k}(t) = 0, \quad k = 1, \ldots, n_\gamma.
\end{equation*}
\end{itemize}

We now combine the individual solvability conditions to formulate a unified criterion for the original boundary value problem \eqref{dae}, \eqref{BC}. We assume that conditions \textnormal{\textbf{(A0)}--\textbf{(A3)}} hold, and fix an arbitrary function $\widetilde{\varphi}(t) \in C([0,T], \mathbb{R}^{n_\beta})$.

The solvability conditions established in parts~(i)–(iv) impose simultaneous requirements on the components $f_i(t)$, $i=0,1,2,3$, of the transformed inhomogeneity $Wf(t)$, as well as on the vectors $\widehat{d}_i$.
The solvability conditions for the four subsystems are summarized as follows:
\begin{itemize}
    \item[(i)] \( \widehat{d}_0 \in  \operatorname{ran} Q_0\),
    \item[(ii)] \( \widehat{d}_1 = 0 \),
    \item[(iii)] \( \widehat{d}_2 \in  \operatorname{ran } Q_2 \) if \( |\beta| > n_\beta \); \quad \( \widehat{d}_2 = 0 \) if \( |\beta| = n_\beta \),
    \item[(iv.1)] \( \widehat{d}_3 = 0 \) if \( |\gamma| > n_\gamma \), 
    \item[(iv.2)] \( \sum\limits_{i=1}^{\gamma_k} \left( f_{\gamma_k,i}(t) \right)^{(\gamma_k - i)} = 0 \), \quad \( k = 1, \ldots, n_\gamma \).
\end{itemize}

Noting that \(\sum_{i=0}^{3}\widehat d_i=\widehat d\), conditions (i)–(iii) together with (iv.1) yield, for the full system,
\[
\widehat d \in \operatorname{ran}\,[\,Q_0\ \ Q_2\,],
\]
which, together with (iv.2), determines the solvability criterion for the general boundary-value problem \eqref{dae}, \eqref{BC}, as stated in conditions~(a) and~(b) of the theorem.

Since $\widehat{d}\in\operatorname{ran}Q$, the linear system $Qv=\widehat{d}$, where $v=\left(\lambda_0^\top,\mu_2^\top\right)^\top\in\mathbb{R}^{n_0+|\beta|-n_\beta}$, is consistent. Its general solution is
\[
v = Q^g\widehat{d} + (I_{n_0+|\beta|-n_\beta} - Q^gQ)\,z, \qquad z \in \mathbb{R}^{n_0+|\beta|-n_\beta},
\]
where $Q^g$ is any generalized inverse of $Q$. Applying $P_0$ and $P_2$ respectively yields the initial values
\[
\lambda_0 = x_0(0) = P_0 v, \qquad \mu_2 = P_2 v,
\]
which, substituted into the variation-of-constants formula for the ODE part and the formula \eqref{x2_gen} for the underdetermined part, give the expressions for $x_0(t)$ and $x_2(t)$ stated in the theorem. Finally, assembling these solutions and \eqref{sol_x1_genBVP}, \eqref{sol_x3_genBVP}, obtained for the respective subproblems, yields the general solution of the boundary value problem in the form \eqref{SOLUTION}. The representations of the components $x_0$ and $x_2$ of the general solution do not depend on the choice of a generalized inverse $Q^g$; see Lemma \ref{lemma_Solution_AE_via_GenInv}. This completes the proof of Theorem~\ref{Th1}.

\end{proof}

\subsection{Case II: $n_0+|\beta| - n_{\beta} = 0$}

In this case, we have \( n_0 = 0 \), which implies that the ODE part is absent in the decomposition \eqref{KCF_1}, and \( |\beta| = n_\beta \). The latter condition occurs in one of two situations: either the underdetermined part is not present in \eqref{KCF_1}, or it is present with all components of the multiindex \( \beta \) equal to 1.

The proof of the following theorem proceeds along the same lines as that of Theorem~\ref{Th1}, with the only difference being that the ODE block is absent and, if the underdetermined part  is present,  it consists solely of components with $\beta_k=1$.

\begin{theorem}\label{Th1_second_case}
Let \( W \) and \( V \) be invertible matrices that transform the matrix pair \( (E, A) \) in \eqref{dae} into the quasi-Kronecker canonical form \eqref{KCF_1} with 
\( n_0 + |\beta| - n_\beta = 0 \).
Suppose that assumptions \textnormal{\textbf{(A1)}--\textbf{(A3)}} are satisfied. 

Then, for any \( \widetilde{\varphi} \in C([0,T],\mathbb{R}^{n_{\beta}}) \), the boundary value problem \eqref{dae}, \eqref{BC} admits a solution \( x \in C^1([0,T],\mathbb{R}^n) \) if and only if the following conditions hold:
\begin{itemize}
    \item[\emph{(a)}] 
        \(
        \sum\limits_{i=1}^{\gamma_k} f_{\gamma_k,i}^{(\gamma_k - i)}(t) = 0,\quad  k = 1,\ldots,n_{\gamma};
        \)
    \item[\emph{(b)}] 
    $
    \widehat{d}_1 =\widehat{d}_2 = \widehat{d}_3 = 0.
    $
\end{itemize}

Any solution \( x \) is given by
\begin{equation}\label{SOLUTION_all_betas_1}
    x(t) = V
\begin{bmatrix}
 -\sum\limits_{i=0}^{\nu_{\alpha} - 1} N_{\alpha}^i f_1^{(i)}(t) \\
 \widetilde{\varphi}(t) \\
 -K_{\gamma} \sum\limits_{i=0}^{\nu_{\gamma} - 1} N_{\gamma}^i f_3^{(i)}(t)
\end{bmatrix}.
\end{equation}

If \( n_\alpha =0 \), then the first component is omitted from \eqref{SOLUTION_all_betas_1}.

If \( n_{\beta} = 0 \), the second component is omitted.

If \( |\gamma| = n_{\gamma} \), the third component is omitted.
\end{theorem}

\subsection{Uniqueness theorem for the boundary value problem \eqref{dae},\eqref{BC}}

\begin{theorem}\label{Theorem_GenBVP_UniqueSolution}
    Let $W$ and $V$ be invertible matrices which transform the matrix pair $(E,A)$ in \eqref{dae} to quasi-Kronecker form \eqref{KCF_1}. Suppose that the assumptions
    \textnormal{\textbf{(A0)}}, \textnormal{\textbf{(A1)}}, and \textnormal{\textbf{(A3)}}
    are satisfied.

        Then, problem \eqref{dae}, \eqref{BC} has a unique solution $x\in C^1([0,T],\mathbb{R}^n)$ if and only if:
         \begin{itemize}
        \item[\emph{(a)}] the block $(L_{\beta},K_{\beta})$ is not present in the matrix pair $(WEV, WAV)$;
        \item[\emph{(b)}] $\sum\limits_{i=1}^{\gamma_k}\left(f_{\gamma_k,i}(t)\right)^{(\gamma_k-i)}=0$, $~~k=1,\ldots,n_{\gamma}$;
         \item[\emph{(c)}] $\widehat{d}\in \textnormal{ran}~ Q$, where $Q$ and $\widehat{d}$ are from \eqref{Q} and \eqref{d_hat};
         \item[\emph{(d)}] $\emph{rank}~ Q=n_0$.
        \end{itemize}

        The unique solution is given by
        \begin{equation}\label{UNIQUE_SOLUTION}
            x(t)=V\begin{bmatrix}
        e^{A_0t}(Q^\top Q)^{-1}Q^\top\widehat{d}_0+\int\limits_0^t e^{A_0(t-s)}f_0(s)ds\\
        -\sum\limits_{i=0}^{\nu_{\alpha}-1}N_{\alpha}^i f_1^{(i)}(t)\\
        -K_{\gamma}\sum\limits_{i=0}^{\nu_{\gamma}-1}N_{\gamma}^i f_3^{(i)}(t)
         \end{bmatrix}.
        \end{equation}

If \( |\gamma| = n_{\gamma} \), then the third component is omitted in \eqref{UNIQUE_SOLUTION}.
\end{theorem}
\begin{proof}
Let conditions~(b) and~(c) of the theorem be satisfied. Then, by Theorem~\ref{Th1}, the boundary value problem \eqref{dae}, \eqref{BC} admits a solution of the form \eqref{SOLUTION}. The third component of this solution involves an arbitrary function $\widetilde{\varphi}(t)$. In order for the problem \eqref{dae}, \eqref{BC} to have a unique solution, this arbitrary component must vanish. This implies that the quasi-Kronecker form of the matrix pair $(E, A)$ must not contain a block of the form $(L_{\beta}, K_{\beta})$.

In this case, the matrix $Q$ reduces to $Q_0$, and the subproblem corresponding to the ODE part must have a unique solution. By Corollary~\ref{corollary_ODEpart_uniqueness}, this holds if and only if $\operatorname{rank} Q = n_0$. Using the representation \eqref{sol_x0_unique} for the first component and removing the third component from the general solution \eqref{SOLUTION}, we conclude that the unique solution of the problem \eqref{dae}, \eqref{BC} is given by \eqref{UNIQUE_SOLUTION}.

\end{proof}

\begin{corollary}\label{Corrolary}
Let $W$ and $V$ be invertible matrices which transform a regular matrix pair $(E,A)$ into the quasi-Weierstrass form \eqref{WCF_1}. Suppose that assumptions \textnormal{\textbf{(A0)}} and \textnormal{\textbf{(A1)}} are fulfilled. Furthermore, let $l=n_0$.

Then problem \eqref{dae}, \eqref{BC} is uniquely solvable if and only if the matrix $Q_0=B_0+C_0 e^{A_0 T}$ is invertible.
The unique solution is given by
\begin{equation}\label{Corollary_solution}
 x(t)=V\begin{bmatrix}
     e^{A_0t}Q_0^{-1}d_*+\int\limits_0^t e^{A_0(t-s)}f_0(s)ds\\
     -\sum\limits_{i=0}^{\nu_{\alpha}-1}N_{\alpha}^i f_1^{(i)}(t)
 \end{bmatrix},
\end{equation}
where
\begin{equation}\label{d*_hat}
d_*=
d-C_0\int\limits_0^T e^{A_0(T-s)}f_0(s)ds+\sum\limits_{i=0}^{\nu_{\alpha}-1}\left(B_1 N_{\alpha}^i f_1^{(i)}(0)+C_1N_{\alpha}^i f_1^{(i)}(T)\right).
\end{equation}
\end{corollary}
\begin{proof} The result follows immediately from Theorem \ref{Theorem_GenBVP_UniqueSolution} if we take into account that the quasi-Weierstrass form for $(E,A)$ contains only the diagonal blocks $(I_{n_0}, A_0)$ and $(N_{\alpha}, I_{|\alpha|})$, and the matrix $Q_0$ is square (by assumption, $l=n_0$).
\end{proof}

\begin{remark}

Assume that the boundary condition $Bx(0)=Bx(T)$ holds. This means
$B=-C$ and thus $B_0=-C_0$. If the
matrix pair $(E,A)$ is in the quasi-Weierstrass form \eqref{WCF_1}, one finds
that $\det Q=\det B_0 \det(I-e^{A_0 T})$. By the spectral theorem, 
$\det(I-e^{A_0 T})=0$ if and only if  $0$ belongs to the spectrum
of $A_0$, that is $\det A_0=0$.  If  $\det A_0\not=0$, the
problem \eqref{dae}, \eqref{BC} is uniquely solvable if and only if the matrix
$B_0$ is invertible, and then the boundary condition $Bx(0)=Bx(T)$ reduces to
$x(0)=x(T)$.

If $(E,A)$ is a purely singular pencil without any regular part,
then $Q=(B_2+C_2 e^{N_{\beta} T})L_\beta^\top$. In that case 
$$
QQ^T=(B_2+C_2 e^{N_{\beta} T})L_\beta^\top L_\beta(B_2^T+e^{N_{\beta}^T T})C_2^T
$$ is not invertible, since $\det (L_\beta^\top L_\beta)=0$. On the other hand,
if the
matrix pair $(E,A)$ is in the quasi-Kronecker form with the the boundary condition $x(0)=x(T)$ such that $\det A_0\not=0$, then $QQ^T$ is invertible since it is positive definite as
$QQ^T\ge Q_0Q_0^T>0$.
\end {remark}

\section{Examples}

\begin{example}
     Consider the differential algebraic equation
  \begin{equation}\label{Ex_reg_syst}
  \begin{aligned}
     \dot{x}_1(t)+x_3(t)&=0,\\
     \dot{x}_2(t)+x_3(t)&=0,\\
     x_2(t)-\sin t&=0,
     \end{aligned}
\end{equation}
subject to the boundary condition
\begin{equation}\label{Ex_reg_BC}
    Bx(0)+Cx(1)=d
\end{equation}
with $B=(b_1,~b_2,~b_3), C=(c_1,~c_2,~c_3)\in\mathbb{R}^{1\times 3}$ and $d\in\mathbb{R}$.
\vskip0.3cm

Here 
\begin{equation*}
    E=\begin{bmatrix}
        1 & 0 & 0\\ 0 & 1 & 0 \\ 0 & 0 & 0
    \end{bmatrix},\quad A=\begin{bmatrix}
        0 & 0 & -1\\ 0 & 0 & -1 \\ 0 & -1 & 0
    \end{bmatrix}, \quad f(t)=\begin{bmatrix}
        0 \\ 0 \\ \sin t
    \end{bmatrix}.
\end{equation*}
    The matrix pair
$(E,A)$ in \eqref{Ex_reg_syst} is regular. The matrices 
\begin{equation*}
    W=\begin{bmatrix}
        1 & -1 & 0\\ 0 & 0 & -1 \\ 0 & 1 & 0
    \end{bmatrix}~~ \text{and} ~~V=\begin{bmatrix}
        1 & 1 & 0\\ 0 & 1 & 0 \\ 0 & 0 & -1
    \end{bmatrix} 
\end{equation*}
transform $(E,A)$ to Weierstrass canonical form
\begin{equation*}
    (WEV,WAV)=(\text{diag}(I_1, N_2),\text{diag}(A_0, I_2))= \left(\left[\begin{array}{c|c c}
        1 & 0 & 0\\ \hline 0 & 0 & 0 \\ 0 & 1 & 0
    \end{array}\right], \left[\begin{array}{c|c c}
        0 & 0 & 0\\\hline 0 & 1 & 0 \\ 0 & 0 & 1
    \end{array} \right]\right).
\end{equation*}

The performed transformation yields
\begin{equation*}
    Wf(t)=\begin{bmatrix}
        f_0(t)\\ f_1(t)
    \end{bmatrix} ~~\text{with} ~~~f_0(t)=0,~ f_1(t)=\begin{bmatrix}
        -\sin t\\ 0 \end{bmatrix},
\end{equation*}
\begin{equation*}
    BV=(B_0,~B_1)~~\text{with} ~~~B_0=(b_1), B_1=(b_1+b_2,~-b_3),
\end{equation*}
\begin{equation*}
    CV=(C_0,~C_1)~~\text{with} ~~~C_0=(c_1), C_1=(c_1+c_2,~-c_3).
\end{equation*}

The matrix $Q_0=B_0+C_0 e^{A_0 T}=b_1+c_1$ has an inverse if $b_1+c_1\neq 0$.
Hence, by Corollary \ref{Corrolary}, the boundary value problem \eqref{Ex_reg_syst}, \eqref{Ex_reg_BC} has a unique solution $x$ if and only if $b_1+c_1\neq 0$. By using \eqref{Corollary_solution}, \eqref{d*_hat} we obtain the unique solution of problem \eqref{Ex_reg_syst}, \eqref{Ex_reg_BC}
\begin{equation*}
    x(t)=V\begin{bmatrix}
        Q_0^{-1}d_*\\-\sum\limits_{i=0}^{1}N_2^i f_1^{(i)}(t)
    \end{bmatrix}=\begin{bmatrix}
        \rho+\sin t\\\sin t\\-\cos t
    \end{bmatrix},
\end{equation*}
where $\rho=(b_1+c_1)^{-1}(d+b_3-(c_1+c_2)\sin 1 +c_3\cos 1)$.
\end{example}
\begin{example} Consider differential-algebraic equation \eqref{dae} with
overdetermined parts, that is
\begin{equation*}
    E=\text{diag}(I_1,L_{(2,1)}^\top)=\left[\begin{array}{c|c}
        1 & 0\\ \hline0 & 0\\ 0 & 1\\0 & 0
    \end{array}\right], A=\text{diag}(A_0,K_{(2,1)}^\top)=\left[\begin{array}{c|c}
        0 & 0\\ \hline0 & 1\\ 0 & 0\\0 & 0
    \end{array}\right], f(t)=\begin{bmatrix}
        3t^2\\ \sin t \\ -\cos t\\0
        \end{bmatrix},
\end{equation*}
which is subject to the boundary condition
\begin{equation}\label{Ex_sing_BC}
    \begin{bmatrix}
        2 & 0\\ 0 & -1
    \end{bmatrix}x(0)+\begin{bmatrix}
        1 & 0\\ 1 & 0
    \end{bmatrix}x(1)=\begin{bmatrix}
        4\\ 2
    \end{bmatrix}.
\end{equation}

We apply Theorem \ref{Theorem_GenBVP_UniqueSolution} to problem \eqref{dae}, \eqref{Ex_sing_BC}. The matrix pair $(E,A)$ is already in Kronecker canonical form. Hence, condition (a) of Theorem \ref{Theorem_GenBVP_UniqueSolution} is met.

For the first block with $n_0=1$, $A_0=0$, we have an ordinary differential equation  \[\dot{x}_1(t)=3t^2.\] The second block of size $3\times 1$ corresponds to the overdetermined part. We have $\gamma=(2,1)$, $|\gamma|=3$, $n_{\gamma}=2$, $\nu_{\gamma}=2$. For problem \eqref{dae}, \eqref{Ex_sing_BC} we have
\begin{equation*}
    B_0=\begin{bmatrix}
        2\\0
    \end{bmatrix}, ~ B_3=\begin{bmatrix}
        0\\-1
    \end{bmatrix}, ~ C_0=\begin{bmatrix}
        1\\1
    \end{bmatrix}, ~ C_3=\begin{bmatrix}
       0\\0
    \end{bmatrix},
\end{equation*}
\begin{equation*}
    f_0(t)=3t^2, ~~f_3(t)=\begin{bmatrix}
        f_{\gamma_1}(t) \\ f_{\gamma_2}(t)
    \end{bmatrix}=\begin{bmatrix}
        \begin{pmatrix}
            \sin t \\ -\cos t
        \end{pmatrix}\\ 0
    \end{bmatrix}.
\end{equation*}

Let us verify condition (b) of Theorem \ref{Theorem_GenBVP_UniqueSolution}. Indeed, we have
\begin{equation*}\begin{array}{ll}
    k=1, \gamma_1=2: &  \sum\limits_{i=1}^2 f_{\gamma_1,i}^{(2-i)}(t)=\left(f_{\gamma_1,1}(t)\right)^\prime+f_{\gamma_1,2}(t)=\left(\sin t\right)^\prime-\cos t=0;\\
     k=2, \gamma_2=1: & \sum\limits_{i=1}^1f_{\gamma_2,i}^{(1-i)}(t)=f_{\gamma_2,1}(t)=0.
\end{array}
\end{equation*}

To check conditions (c) and (d) of Theorem \ref{Theorem_GenBVP_UniqueSolution}, we begin with calculating the matrix $Q$ and the vector $\widehat{d}$:
\begin{equation*}
    Q=B_0+C_0 e^{A_0 T}=\begin{bmatrix}
        2\\0
    \end{bmatrix}+\begin{bmatrix}
        1\\1
    \end{bmatrix}e^{0}=\begin{bmatrix}
        3\\1
    \end{bmatrix}
\end{equation*}
and condition (d) of Theorem \ref{Theorem_GenBVP_UniqueSolution} is fulfilled.
Moreover,
\begin{equation*}\begin{aligned}
    \widehat{d}&=d-C_0\int\limits_0^1 e^{A_0(1-s)}f_0(s)ds\\&+\sum\limits_{i=0}^1\left(B_3K_{(2,1)}N_{(2,1)}^i f_3^{(i)}(0)+C_3K_{(2,1)}N_{(2,1)}^i f_3^{(i)}(T)\right)=\begin{bmatrix}
        3\\1
    \end{bmatrix},
\end{aligned}
   \end{equation*}
 which shows (c) of Theorem \ref{Theorem_GenBVP_UniqueSolution}.
 Hence, problem \eqref{dae}, \eqref{Ex_sing_BC} has a unique solution $x$, which we obtain from \eqref{UNIQUE_SOLUTION}
\begin{equation*}
    x(t)=\begin{bmatrix}
        (Q^\top Q)^{-1}Q^\top\widehat{d}+\int\limits_0^t 3s^2 ds\\
        -K_{(2,1)}\sum\limits_{i=0}^1 N_{(2,1)}^i \begin{bmatrix}
            \sin t\\-\cos t\\0
        \end{bmatrix}^{(i)}
         \end{bmatrix}=\begin{bmatrix}
             t^3+1\\-\sin t
         \end{bmatrix}.
\end{equation*}

\end{example}

\begin{example}
    Consider the boundary value problem
    \begin{equation}\label{Example_under}
    L_{(3,1,2)}\dot{x}(t)=K_{(3,1,2)} x(t)+\left(
        12t^2, 1, 2t\right)^\top,
\end{equation}
\begin{equation}\label{Example_under_BC}
   B_2x(0)+C_2x(1)=\begin{bmatrix}
        1 \\  1
    \end{bmatrix}
\end{equation}
with
\begin{equation*}
   B_2=\begin{bmatrix}
        0 & 0 & 0 & 1 & 0 & 0\\ 1 & 0 & -1 & 0 & 0 & 1
    \end{bmatrix},\quad C_2=\begin{bmatrix}
        0 & 0 & 0 & 0 & 0 & 1\\ 0 & 0 & 0 & 0 & -1 & 0
    \end{bmatrix}.
\end{equation*}

The differential-algebraic equation \eqref{Example_under} is already in its Kronecker canonical form, consisting solely of the underdetermined block, corresponding to the multi-index $\beta=(3,1,2)$.

Hence, given an arbitrary function $\widetilde{\varphi}=\left(\varphi_1,\varphi_2,\varphi_3\right)^\top\in C([0,1],\mathbb{R}^3)$, we need to verify condition (b) of Theorem \ref{Th1}. To this end, we compute
$Q$ and $\widehat{d}$:
\begin{equation*}\label{Q_ex_UNDET}
    Q=\begin{bmatrix} B_2+C_2e^{N_{\beta}}
\end{bmatrix}L_{\beta}^\top=\begin{bmatrix}
            0&0&1\\0&-1&1
        \end{bmatrix},
\end{equation*}
\begin{equation*}\label{EX-d-Undet}
        \widehat{d}=
d-B_2M_{\beta}\widetilde{\varphi}(0)-C_2 \widetilde{h}_{\beta}(1; \widetilde{\varphi}(1))=\begin{bmatrix}
    -\varphi_2(0)-\int\limits_0^1\varphi_3(s)ds\\1-\varphi_1(0)+\varphi_3(1)\end{bmatrix}.
\end{equation*}

 Since $\text{ran }Q=\mathbb{R}^2$, condition (b) of Theorem \ref{Th1} is satisfied for any vector $\widehat{d}\in \mathbb{R}^2$.  Hence,  problem \eqref{Example_under}, \eqref{Example_under_BC} admits a solution $x$ for any arbitrarily chosen $\widetilde{\varphi}$. It follows from \eqref{SOLUTION} that $x$ is determined by
\begin{equation}\label{Ex_und_x_mu}
    x(t)=e^{N_{\beta}t}L_{\beta}^\top\left(Q^g\widehat{d}+(I_3-Q^gQ)z\right) + \widetilde{h}_{\beta}(t; \widetilde{\varphi}(t)),
   \end{equation}
where $Q^g$ is a generalized inverse of $Q$, and $z$ is an arbitrary vector in $\mathbb{R}^3$.

Since $Q$ is a full row rank matrix ($\text{rank}~Q=2$), its right inverse can serve as a generalized inverse, see \eqref{right_inverse}: 
\begin{equation*}
Q^g=Q^\top\left(QQ^\top\right)^{-1}=\begin{bmatrix}
        0&0\\1&-1\\1&0
    \end{bmatrix}.
\end{equation*}

Let us set, for instance, $\widetilde{\varphi}(t)=(0, 0, 0)^\top$ and $z=(0, 0, 0)^\top$. Then
\begin{equation*}
    \widetilde{h}_{\beta}(t; \widetilde{\varphi}(t))=\int\limits_0^t e^{N_\beta(t-s)}L_\beta^\top\begin{bmatrix}
        12s^2,1,2s
    \end{bmatrix}^\top ds=\begin{bmatrix}
    0,4t^3,t^4+t,0,0,t^2
\end{bmatrix}^\top,~~\widehat{d}=\begin{bmatrix}
         -2\\0
     \end{bmatrix}.
\end{equation*}

Substituting determined $Q^g$, $\widetilde{h}_{\beta}(t; \widetilde{\varphi}(t))$, and $\widehat{d}$ into \eqref{Ex_und_x_mu}, we obtain that, for the chosen $\widetilde{\varphi}(t)$ and $z$, problem \eqref{Example_under}, \eqref{Example_under_BC} has a solution
\begin{equation*}
     x(t)=e^{N_{\beta}t}L_{\beta}^\top Q^g\widehat{d} + \widetilde{h}_{\beta}(t; \widetilde{\varphi}(t))=\begin{bmatrix}
         0, 4t^3, t^4+t-1, 0, 0, t^2
     \end{bmatrix}^\top.
   \end{equation*}
\end{example}


\begin{thebibliography}{99}

\bibitem{Assanova_Trunk_Uteshova_2024}
A. Assanova, C. Trunk, R. Uteshova, On the solvability of boundary value problems for linear differential-algebraic equations with constant coefficients, Contemp. Math. 798 (2024) 13–19. https://doi.org/10.1090/conm/798.

\bibitem{ascher1992projected}
U. Ascher, L. Petzold, Projected collocation for higher-order higher-index differential-algebraic equations, J. Comput. Appl. Math. 43 (1992) 243–259. https://doi.org/10.1016/0377-0427(92)90269-4.

\bibitem{ascher1998computer}
U. Ascher, L. Petzold, Computer Methods for Ordinary Differential Equations and Differential-Algebraic Equations, vol. 61, SIAM, Philadelphia, 1998.

\bibitem{bai1992modified}
Y. Bai, A modified Lobatto collocation for linear boundary value problems of differential-algebraic equations, Comput. 49 (1992) 139–150. https://doi.org/10.1007/BF02238746.

\bibitem{benisrael2006GeneralizedInverses}
A. Ben-Israel, T.N. Greville, Generalized Inverses: Theory and Applications, Springer, Berlin, 2006.

\bibitem{berger_diss}
T. Berger, On differential-algebraic control systems, TU Ilmenau, 2013.

\bibitem{berger2016linear}
T. Berger, C. Trunk, H. Winkler, Linear relations and the Kronecker canonical form, Linear Algebra Appl. 488 (2016) 13–44. https://doi.org/10.1016/j.laa.2015.09.033.

\bibitem{BT12}
T. Berger, S. Trenn, The quasi-Kronecker form for matrix pencils, SIAM J. Matrix Anal. Appl. 33 (2012) 336–368. https://doi.org/10.1137/110826278.

\bibitem{Berger2013addition}
T. Berger, S. Trenn, Addition to “The quasi-Kronecker form for matrix pencils”, SIAM J. Matrix Anal. Appl. 34 (2013) 94–101. https://doi.org/10.1137/120883244.

\bibitem{brenan1995numerical}
K. Brenan, S. Campbell, L. Petzold, Numerical Solution of Initial-Value Problems in Differential-Algebraic Equations, SIAM, Philadelphia, 1995.

\bibitem{clark1989numerical}
K. Clark, L. Petzold, Numerical solution of boundary value problems in differential-algebraic systems, SIAM J. Sci. Stat. Comput. 10 (1989) 915–936. https://doi.org/10.1137/0910053.

\bibitem{Dzhumabaev:1989}
D. Dzhumabaev, Criteria for the unique solvability of a linear boundary-value problem for an ordinary differential equation, USSR Comput. Math. Math. Phys. 29 (1989) 34–46. https://doi.org/10.1016/0041-5553(89)90038-4.

\bibitem{gantmacher1959theory}
F.R. Gantmacher, The Theory of Matrices, vol. I \& II, Chelsea, New York, 1959.

\bibitem{kronecker1890algebraische}
L. Kronecker, Algebraische Reduktion der Schaaren bilinearer Formen, Sitzungsber. Akad. Wiss. Berlin (1890) 763–767.

\bibitem{Kunkel:2006}
P. Kunkel, V. Mehrmann, Differential-Algebraic Equations: Analysis and Numerical Solution, Eur. Math. Soc., Zürich, 2006.

\bibitem{lamour2013differential}
R. Lamour, R. März, C. Tischendorf, Differential-Algebraic Equations: A Projector Based Analysis, Springer, Berlin, 2013.

\bibitem{Lamour_Maerz2015}
R. Lamour, R. März, E. Weinmüller, Boundary-value problems for differential-algebraic equations: A survey, in: A. Ilchmann, T. Reis (Eds.), Surv. Differ. Algebraic Equ. III, Differ.-Algebraic Equ. Forum, Springer, Berlin, 2015, pp. 177–309.

\bibitem{LancasterRodman2005}
P. Lancaster, L. Rodman, Canonical forms for Hermitian matrix pairs under strict equivalence and congruence, SIAM Rev. 47 (2005) 407–443. https://doi.org/10.1137/S003614450444556X.

\bibitem{marz1996canonical}
R. März, Canonical projectors for linear differential algebraic equations, Comput. Math. Appl. 31 (1996) 121–135. https://doi.org/10.1016/0898-1221(95)00224-3.

\bibitem{marz2004solvability}
R. März, Solvability of linear differential algebraic equations with properly stated leading terms, Result. Math. 45 (2004) 88–105. https://doi.org/10.1007/BF03323000.

\bibitem{Puche_Reis_2018_ConstantcoefficientDAE}
M. Puche, T. Reis, F.L. Schwenninger, Constant-coefficient differential-algebraic operators and the Kronecker form, Linear Algebra Appl. 552 (2018) 29–41. https://doi.org/10.1016/j.laa.2018.04.005.

\bibitem{RabierRheinboldt02}
P.J. Rabier, W.C. Rheinboldt, Theoretical and numerical analysis of differential-algebraic equations, in: B.S. Garbow (Ed.), Solution of Equations in $\mathbb{R}^n$ (Part 4), Techniques of Scientific Computing (Part 4), North-Holland, Amsterdam, 2002, pp. 183–540.

\bibitem{riaza2008differential}
R. Riaza, Differential-Algebraic Systems: Analytical Aspects and Circuit Applications, World Sci. Publ., Singapore, 2008.


\bibitem{stover2001collocation}
R. Stöver, Collocation methods for solving linear differential-algebraic boundary value problems, Numer. Math. 88 (2001) 771–795. https://doi.org/10.1007/PL00005458.

\bibitem{Trenn13}
S. Trenn, Solution concepts for linear DAEs: a survey, in: A. Ilchmann, T. Reis (Eds.), Surv. Differ. Algebraic Equ. I, Springer, Berlin, 2013, pp. 137–172.

\end{thebibliography}

\end{document}